\documentclass[reqno,a4paper,12pt,oneside]{amsart} 

\ifdefined\SMART
\usepackage[paperwidth=12cm,paperheight=24cm,left=5mm,right=5mm,top=10mm,bottom=5mm]{geometry}
\else
\calclayout
\fi

\usepackage{amsmath,amscd,amsfonts,amssymb}
\usepackage{mathrsfs,dsfont}
\usepackage{hyperref}

\numberwithin{equation}{section}
\numberwithin{figure}{section}

\usepackage{txfonts,pxfonts,tikz} 
\usepackage{tgschola}
\usepackage[T1]{fontenc}

\newcommand\R{\mathbb{R}}

\newcommand\Z{\mathbb{Z}}

\newcommand\al{\alpha}
\newcommand\be{\beta}
\newcommand\gam{\gamma}
\newcommand\Gam{\Gamma}

\newcommand\lam{\lambda}
\newcommand\Lam{\Lambda}

\newcommand\Om{\Omega}

\newcommand\1{\mathds{1}}

\renewcommand\le{\leqslant}
\renewcommand\ge{\geqslant}

\newcommand\sbt{\subset}

\newcommand{\ft}[1]{\widehat{#1}}
\newcommand{\dotprod}[2]{\langle #1 , #2 \rangle}

\newcommand{\half}{\tfrac{1}{2}}

\theoremstyle{plain}
\newtheorem{thm}{Theorem}[section]
\newtheorem{lem}[thm]{Lemma}
\newtheorem{lemma}[thm]{Lemma}
\newtheorem{cor}[thm]{Corollary}

\newtheorem{problem}[thm]{Problem}

\newtheorem*{claim*}{Claim}

\newcommand{\thmref}[1]{Theorem~\ref{#1}}
\newcommand{\secref}[1]{Section~\ref{#1}}

\newcommand{\lemref}[1]{Lemma~\ref{#1}}

\newcommand{\probref}[1]{Problem~\ref{#1}}
\newcommand{\corref}[1]{Corollary~\ref{#1}}

\theoremstyle{definition}

\newtheorem*{definition*}{Definition}
\newtheorem*{remarks*}{Remarks}
\newtheorem*{remark*}{Remark}

\newenvironment{enumerate-roman}
{\begin{enumerate}
\addtolength{\itemsep}{5pt}
}
{\end{enumerate}}

\newenvironment{enumerate-alph}
{\begin{enumerate}
\addtolength{\itemsep}{5pt}
}
{\end{enumerate}}

\newenvironment{enumerate-num}
{\begin{enumerate}
\addtolength{\itemsep}{5pt}
}
{\end{enumerate}}

\newenvironment{enumerate-text}
{\begin{enumerate}
\addtolength{\itemsep}{5pt}
}
{\end{enumerate}}

\newcommand{\beql}[1]{\begin{equation}\label{#1}}
\newcommand{\eeq}{\end{equation}}

\newcommand{\RR}{{\mathbb R}}

\newcommand{\ZZ}{{\mathbb Z}}

\newcommand{\TT}{{\mathbb T}}

\begin{document}

\ifdefined\SMART
\title[Maximality, completeness of exponentials]{Maximality and completeness of orthogonal exponentials on the cube}
\else
\title{Maximality and completeness of orthogonal exponentials on the cube}
\fi

\author[M. Kolountzakis]{Mihail N. Kolountzakis}
\address{Mihail N. Kolountzakis: \newline Department of Mathematics and Applied Mathematics, University of Crete, Voutes Campus, 70013 Heraklion, Greece \newline and \newline Institute of Computer Science, Foundation of Research and Technology Hellas, N. Plastira 100, Vassilika Vouton, 700 13, Heraklion, Greece}
\email{kolount@gmail.com}

\author[N. Lev]{Nir Lev}
\address{Nir Lev: \newline Department of Mathematics, Bar-Ilan University, Ramat-Gan 5290002, Israel}
\email{levnir@math.biu.ac.il}

\author[M. Matolcsi]{M\'at\'e Matolcsi}
\address{M\'at\'e Matolcsi: \newline HUN-REN Alfr\'ed R\'enyi Mathematical Institute Re\'altanoda utca 13-15, H-1053, Budapest, Hungary and Department of Analysis \newline and \newline Operations Research, Institute of Mathematics, Budapest University of Technology and Economics, M\H uegyetem rkp. 3., H-1111 Budapest, Hungary.}
\email{matomate@renyi.hu}

\dedicatory{Dedicated to the memory of Bent Fuglede}
\date{October 16, 2024}
\thanks{N.L.\ is supported by ISF Grant No.\ 1044/21.
M.M.\ is supported by the Hungarian National Foundation for Scientific Research, Grants No. K132097, K146387.}

\begin{abstract}
It is possible to have a packing by translates of a cube that is maximal (i.e.\ no other cube can be added without overlapping) but does not form a tiling. In the long running analogy of packing and tiling to orthogonality and completeness of exponentials on a domain, we pursue the question whether one can have maximal orthogonal sets of exponentials for a cube without them being complete. We prove that this is not possible in dimensions $1$ and $2$, but is possible in dimensions $3$ and higher. We provide several examples of such maximal incomplete sets of exponentials, differing in size, and we raise relevant questions. We also show that even in dimension $1$ there are sets which are spectral (i.e.\ have a complete set of orthogonal exponentials) and yet they also possess maximal incomplete sets of orthogonal exponentials.
\end{abstract}

\maketitle


\section{Introduction}

\subsection{}
If $\Omega \subset \RR^d$ is a measurable set of finite measure, we call it \textit{spectral} if there exists a countable set $\Lambda \subset \RR^d$ such that the  system 
 of exponential  functions
\begin{equation}
\label{eqEXPSYS}
E(\Lam) = \{ e^{2 \pi i \dotprod{\lam}{x}} : \lam \in \Lam\}
\end{equation} 
forms an orthogonal basis in $L^2(\Omega)$. The set of frequencies $\Lambda$ is then called a \textit{spectrum} of $\Omega$. The properties of spectral sets, especially in comparison with properties of sets that tile by translations, have been a subject of intense research for decades (see, for instance, the recent survey \cite{kolountzakis2024orthogonal}). The Fuglede conjecture \cite{fuglede1974operators} stated that the spectral sets are precisely those that can tile the   space by translations (this means that one can find a collection of translates of the set such that almost every point of the space belongs to exactly one translate), but this conjecture is now known to be false in dimensions $3$ and higher \cite{tao2004fuglede}, \cite{matolcsi2005fuglede4dim}, \cite{kolountzakis2006tiles}, \cite{farkas2006onfuglede}, \cite{farkas2006tiles}. Still a lot more is known and continues to be discovered about the connection of spectrality to tiling, a major recent result being the truth of the Fuglede conjecture for convex sets \cite{lev2022fuglede}.

For the unit cube $Q = [-\half, \half]^d$, which is of course both a spectral set and a translational tile in $\RR^d$, a lot more is known. In particular it is known that the spectra $\Lambda$ of $Q$ are precisely the tiling complements of $Q$, namely the translation sets $\Lambda \subset \RR^d$ such that $\{Q+\lambda\}$,  $\lambda \in \Lambda$, constitutes a tiling \cite{lagarias2000orthonormal}, \cite{iosevich1998spectral}, \cite{kolountzakis2000packing}.

For any set $\Omega \subset \RR^d$ the two exponentials $e^{2\pi i \dotprod{\lam}{x}}$ and $e^{2\pi i \dotprod{\mu}{x}}$ are easily seen to be orthogonal in $L^2(\Omega)$ if and only if $\ft{\1}_\Omega(\lambda-\mu) = 0$,
where $\ft{\1}_\Omega(\xi)
= \int_{\Omega} 
e^{- 2\pi i \dotprod{\xi}{x}} dx$
is the Fourier transform of the
  indicator function $\1_\Om$.
  Since the unit cube $Q$ is a product set, we can easily compute  $\ft{\1}_Q$ and find that its zero set is  
\begin{equation}
\label{zero-set}
G = \{(\xi_1, \xi_2, \ldots, \xi_d) \in \RR^d: \text{there is $j$ such that $\xi_j \in \ZZ\setminus\{0\}$} \}.
\end{equation}
At the same time a collection of translates $\{Q+t\}$,  $t \in T$, forms a packing (i.e.\ the translates overlap at measure zero only) if and only if $(T-T) \cap (-1, 1)^d = \{0\}$. Hence if $\Lambda$ is an orthogonal set for the cube $Q$  then $\Lambda-\Lambda \subset  G \cup \{0\}$, which in turn implies, via the characterization \eqref{zero-set},  that $\{Q+\lambda\}$, $\lambda \in \Lambda$, is a packing. In short, any orthogonal set for $Q$ is also a packing set for $Q$ (note, however, that the converse is not true).

\subsection{}
We now come to the question of maximality. A packing of $Q$ is called \textit{maximal} if it is not possible to add another translated copy of $Q$ so that it remains a packing. Clearly every tiling is a maximal packing. It is easy to see though that there exist maximal packings of $Q$ by a set of translates $T$ which are not tilings. For instance, in dimension one, consider $T$ to contain all numbers of the form $\pm (n-\frac1{4})$, $n=1,2,3,\dots$. The same is true in all dimensions.

An orthogonal set of frequencies $\Lambda$ is similarly called maximal if it is not possible to add another point to it so that it remains orthogonal. Any spectrum of $Q$ (i.e.\ an orthogonal and complete set of frequencies) is of course maximal. But now it is not clear whether maximality fails to imply completeness, just as maximality of a packing fails to imply tiling.

In fact, it is easy to see that in dimension one, the maximality of an orthogonal set of frequencies does imply completeness. Indeed, if the exponential system $E(\Lam)$
is orthogonal in $L^2([-\half, \half])$ then $\Lambda-\Lambda \subset \ZZ$. Fix $\lambda_0 \in \Lambda$. It follows that $\Lambda \subset \lambda_0+\ZZ$ and, since all frequencies in $\lambda_0+\ZZ$ are orthogonal to each other, maximality of $\Lambda$ implies that $\Lambda = \lambda_0+\ZZ$, that is, $\Lambda$ is a spectrum of $[-\half, \half]$.

Is the same still true in higher dimensions? That is, if $\Lam$ is a maximal orthogonal set for the unit cube in $\R^d$, must it be also complete?
In \S\ref{sec:2d} we show that every maximal orthogonal set for the square in $2$ dimensions is also complete, just as in dimension $1$. Then in \S\ref{sec:thick} we show that this is not the case in dimension $3$. The maximal incomplete set we find there has the maximum possible density. In \S\ref{sec:thin} we give an alternative construction, again in dimension $3$, where our maximal set is now much thinner, almost a finite union of planes. We then show in \S\ref{sec:higher} how we can extend the results from dimension $3$ to higher dimensions. In \S\ref{sec:properties} we examine some general properties that maximal orthogonal sets for the cube must have. Closing, in \S\ref{sec:no-cube} we show that if we are willing to examine spectral sets other than the cube then even in dimension $1$ we can find an example of a spectral set which has a maximal incomplete set of exponentials.


\section{Maximal orthogonal sets in two dimensions}
\label{sec:2d}

First we show that in two dimensions, the maximality of an orthogonal set for the unit square implies completeness, similar to the one-dimensional case.

\begin{thm}
\label{thm2DIM}
Any maximal orthogonal set  for the unit square
in $\R^2$ is also complete,  i.e.\ it is a 
spectrum. More generally,  any orthogonal set
for the unit square  can be embedded as a
subset of some spectrum.
\end{thm}

The proof will require
a lemma which
can be found in 
 \cite[Observation 1]{kolountzakis2000structure} as well as in \cite[Lemma 11.4]{greenfeld2017fuglede}.

\begin{lem} \label{lemG5}
Let $X$ be a subset of an abelian group $H$, and let $H_1$ and $H_2$ be two subgroups of $H$. Assume that $X-X\subset H_1\cup H_2$. Then $X-X\subset H_1$ or $X-X\subset H_2$. 
\end{lem}

Now let $\Lam \sbt \R^2$
be an orthogonal set for the unit square.
Then  $\Lam - \Lam \sbt G \cup \{0 \}$,
where $G$ is the zero set from \eqref{zero-set}
(with $d=2$). It follows that if we define
$H_1 = \Z \times \R$ and $H_2 = \R \times \Z$,
which are both subgroups of $\R^2$, then 
$\Lam - \Lam \sbt H_1 \cup H_2$.
\lemref{lemG5} then implies that
 $\Lam  - \Lam \subset \Z \times \R$ or 
$\Lam  - \Lam \subset \R \times \Z$.

Let us consider the case where
 $\Lam  - \Lam \subset \Z \times \R$
 (the other case is similar).
By translating the set $\Lambda$ we may assume that it contains the origin, so this  implies that
$\Lam \subset \Z \times \R$.
We now observe that if
  $(n,t)$ and $(n,s)$ are two distinct
 points in $\Lam$, then orthogonality implies 
 that $t-s$ must be an integer. Hence 
 $\Lam$ is contained in a set of the form
 \begin{equation}
 \label{eq:squarespectrum}
     \{(n, k + t(n)) : n,k \in \Z \}
 \end{equation}
where $t(n)$ are real numbers.
But it is known, see
\cite[Theorem 5]{JP99}, that
any set of the form
 \eqref{eq:squarespectrum}
is a spectrum of the unit square.
 Hence $\Lam$ can be embedded as a
subset of a spectrum.
In particular, $\Lam$ is maximal
if and only if it coincides with 
 a  spectrum.


\section{A maximal incomplete orthogonal set 
in three dimensions}
\label{sec:thick}

\subsection{}
Next we show that \thmref{thm2DIM} does not extend to three dimensions. Namely, we will prove the following result.

\begin{thm}
\label{thm:thick}
There exists an
orthogonal set $\Lam \sbt \R^3$
for the unit cube,
which is maximal but incomplete.
\end{thm}

Our strategy is to start from the integer lattice $\Z^3$ which is a spectrum for the unit cube. We then remove a small part of this spectrum, contained in the union of the three coordinate planes, which leaves us with a subset $B$ of the integer lattice. The set $B$ is still orthogonal, but of course it is not maximal as it embeds in a spectrum. However, we will show that one can add a small set $A$ to $B$ such that the orthogonality is still preserved, but at the same time it eliminates the possibility of further extending $A \cup B$ to a spectrum, or stronger, to any larger orthogonal set.

\subsection{}
We now turn to the details of the construction.
Fix three real numbers $\al, \be, \gam$ such that none of them is an integer. 
Let $A$ be the set in $\R^3$ consisting of all points that have 
one of the following forms:
\begin{equation}
(0, \beta-k, \gam),  \quad (\alpha, 0, \gam-k), \quad (\alpha-k, \beta, 0)
\end{equation}
where $k$ is a nonzero integer.
 Notice that $A$ is contained in the union of three lines.

Let $B$ be the set of all vectors
in $\R^3$ whose coordinates are 
nonzero integers. 
This set contains all integer vectors
except those lying in a union of three planes.

The sets $A$ and $B$ are disjoint.
Let $\Lam = A \cup B$ be their union.
It is straightforward to check that 
$\Lam -  \Lam \sbt G \cup \{0\}$, 
where $G$ is the zero set from \eqref{zero-set}
(with $d=3$). Hence
$\Lam$ is an orthogonal set for the unit cube $Q$ in $\R^3$. In particular, $Q + \Lam$ is a packing.

On the other hand we claim that $\Lam$ is not a complete set of frequencies. Indeed, if it was complete  then $Q + \Lam$ would be a tiling, see
 \cite{lagarias2000orthonormal}, \cite{iosevich1998spectral}, \cite{kolountzakis2000packing}.
 But notice that each one of the three slabs
\begin{equation}
[-\half, \half] \times \R \times \R, \quad
\R \times [-\half, \half] \times \R, \quad
\R \times \R \times [-\half, \half]
\end{equation}
is scarcely covered by the cubes in the packing $Q + \Lam$, so this packing is not a tiling.

\subsection{}
To continue we need the following lemma. 

\begin{lem}
\label{lem:DBZ}
Let $s  \in \R^3$ and suppose that $s - B \sbt G$. 
Then at least one of the coordinates of $s$ is zero. 
\end{lem}

\begin{proof}
Let $s = (u,v,w)$ and 
suppose to the contrary that all the coordinates $u,v,w$ are nonzero. 
Since we have $(u,v,w) - (1,1,1) \in G$ then at least one of the 
coordinates $u,v,w$ must be an integer. Hence, say, $u=j$
for some nonzero integer $j$. Next, since
we also have $(u,v,w) - (j,1,1) \in G$,
then at least one of the other two coordinates $v,w$
must be an integer. Hence, say, $v=k$
for some nonzero integer $k$. Finally,  since
$(u,v,w) - (j,k,1) \in G$,
then $w$ must be an integer, so $w = l$ 
for some nonzero integer $l$. But this implies that
the zero vector
$(u,v,w) - (j,k,l)$ is in $G$, a contradiction.
\end{proof}

\subsection{}
Finally we show that $\Lam$
is a maximal orthogonal  set for the unit cube $Q$.
Suppose to the contrary that this is not the case,
and let
 $s = (u,v,w) \in \R^3$ be
a point not belonging to $\Lam$,
such that $\Lam \cup \{s\}$ is still an orthogonal set. In other words, this means that we have
$s - \Lam  \sbt G$.

In particular, we have $s - B \sbt G$, hence by \lemref{lem:DBZ}
at least one of the coordinates $u,v,w$ is zero. 
Let us consider the case where $u=0$ 
(the other two cases are similar).
We then have $(0,v,w)  - (0,\beta-k,\gam) \in G$ for every
nonzero integer $k$. This  implies that 
either $w - \gam$ is a nonzero integer,  or that $v = \beta$.

If $v = \beta$, then $s = (0,\beta,w)$. 
We now have $(0,\beta,w)  - (\alpha, 0, \gam-k) \in G$ for every 
nonzero integer $k$. This  implies that $w = \gam$.
We thus obtain $s = (0, \beta, \gam)$, but this
contradicts the condition that
$s - (\alpha-k, \beta, 0) \in G$ for every 
nonzero integer $k$. This case cannot therefore happen.

Hence we must have $w = \gam - l$ where $l$
is a nonzero integer. Then $s = (0, v, \gam - l)$.
The condition    $s  - (\alpha, 0, \gam-l) \in G$ 
now implies that $v = m$ for some nonzero integer $m$.
We therefore  obtain that $s = (0, m, \gam - l)$, but this again
contradicts the condition that
$s - (\alpha-k, \beta, 0) \in G$ for every 
nonzero integer $k$. Hence 
this case cannot happen either.

This shows that $\Lam$ is indeed 
 a maximal  orthogonal set for the unit cube.


\section{A ``thin'' maximal set in three dimensions}
\label{sec:thin}

\subsection{}
The maximal incomplete orthogonal set
constructed in \secref{sec:thick}
is, in a sense, almost as large as a 
spectrum. Indeed, this set contains
all integer vectors except those
lying in the union of the three 
coordinate planes, and
in particular, the set has 
the maximum possible density.

In this section we give an alternative
construction which, somewhat surprisingly,
produces a much thinner set.

\begin{thm}
\label{thm:thin}
There is an orthogonal 
set $\Lam \sbt \R^3$ for the unit cube,
which is maximal but
 contained in a finite union of 
  planes (and, as a consequence, is incomplete).
\end{thm}

\subsection{}
Fix  again three real numbers $\al, \be, \gam$ such that none of them is an integer. 
Let $\Lam$ be the set in $\R^3$ consisting of all points that have 
one of the following forms:
\begin{equation}
\label{eq:thinpoints}
(0,0,0), \quad (n, \beta-k, \gam),  \quad (\alpha, n, \gam-k), \quad (\alpha-k, \beta, n)
\end{equation}
where $n$ and $k$ are nonzero integers.

 It is straightforward to check that 
$\Lam -  \Lam \sbt G \cup \{0\}$, hence
$\Lam$ is an orthogonal set for the unit cube $Q$  in $\R^3$. On the other hand, $\Lam$ is contained in the union of three planes plus the origin,
hence  $Q + \Lam$ is  not a tiling 
and  $\Lam$ is not a spectrum for $Q$.

\subsection{}
On the other hand we now claim that $\Lam$
is a maximal orthogonal set for the unit
cube. Indeed, if not then there is 
a point $s = (u,v,w) \in \R^3$ 
not belonging to $\Lam$,
such that $\Lam \cup \{s\}$ is still an
orthogonal set. In other words, this means that  
$s - \Lam  \sbt G$.

The condition $(u,v,w) - (0,0,0) \in G$ implies that 
 at least one of $u,v,w$ is a nonzero integer.
Let us consider the case where $u=j$ for some nonzero integer $j$
(the other two cases are similar).
We then have $(j,v,w)  - (j,\beta-k,\gam) \in G$ for every
nonzero integer $k$. This  implies that 
either $w - \gam$ is a nonzero integer,  or that $v = \beta$.

If $v = \beta$, then $s = (j,\beta,w)$. 
We have $(j,\beta,w)  - (\alpha, n, \gam-k) \in G$ for every two
nonzero integers $n$ and $k$. This  implies that $w = \gam$.
We thus obtain $s = (j, \beta, \gam)$, but this
contradicts the condition that
$s - (\alpha-k, \beta, n) \in G$ for every two
nonzero integers $n$ and $k$. So this case cannot happen.

Hence we must have $w = \gam - l$ where $l$
is a nonzero integer. Then $s = (j, v, \gam - l)$.
The condition $s  - (\alpha, n, \gam-l) \in G$ for every 
nonzero integer $n$ now implies that $v = 0$. 
We therefore  obtain that $s = (j, 0, \gam - l)$, but this again
contradicts the condition that
$s - (\alpha-k, \beta, n) \in G$ for every two
nonzero integers $n$ and $k$. It follows that 
this case cannot happen either.

This shows that $\Lam$ is indeed 
 a maximal  orthogonal set for the unit cube.

\subsection{}
We leave open the following question, which arises naturally.

\begin{problem}
\label{prob:lines}
Does there exist a  maximal 
orthogonal set $\Lam \sbt \R^3$ for the unit cube,
 such that $\Lam$ is contained in
a union of finitely many lines?
\end{problem}


\section{Extensions to higher dimensions}
\label{sec:higher}

We will now use 
the results of Sections \ref{sec:thick} and \ref{sec:thin} in order to construct examples of maximal  incomplete orthogonal sets for the cube also in dimensions greater than $3$.
We will present two methods that allow us to
 use lower-dimensional examples in order to produce examples in higher dimensions.

\subsection{}
Suppose that  $\Lam \sbt \R^n$ is a maximal 
 and incomplete orthogonal set for
 the unit cube in $\R^n$.
We will show how to use this set $\Lam$ in order to
construct a set $\Gam  \sbt \R^n \times \R^m$
 which is a maximal 
 and incomplete  orthogonal set for the unit cube
 in $\R^n \times \R^m$.

We take $\Gam = A \cup B$ to be
the union of two disjoint sets $A$ and $B$,
defined as follows. Let $A$ be 
 the set of all vectors 
 in $\R^n \times \R^m$  of the form
$(\lam, 0)$ where $\lam \in \Lam$.
We also let $B$ be the set of vectors
$(p,q) \in \Z^n \times \Z^m$
such that the vector $q$
has at least one nonzero coordinate,
i.e.\ $q$ is not the zero vector.

It is clear that $A$ and $B$ are indeed two
disjoint sets. It is also straightforward to check that their union $\Gam = A \cup B$ 
is an orthogonal set for the unit cube
 in $\R^n \times \R^m$.
 
\begin{thm}
If $\Lam \sbt \R^n$ is a maximal 
 and incomplete orthogonal set for
 the unit cube in $\R^n$, then
 the set $\Gam$ constructed above
  is a maximal 
 and incomplete  orthogonal set for the unit cube
 in $\R^n \times \R^m$.   
\end{thm}

 \begin{proof}
 Let us first show that $\Gam$ is maximal.
 Suppose to the contrary that this is not the case,
 so there is  a point  $(u,v) \in \R^n \times \R^m$
 which is orthogonal to both $A$ and $B$.
If $v$ is the zero vector, then 
  orthogonality  to  $A$  implies that
the vector $u$ is orthogonal to $\Lam$
(with respect to the
unit cube in $\R^n$) which
contradicts the maximality of  $\Lam$.
Hence this cannot happen, so
 $v$ must be a nonzero vector.
We now define a vector 
$q = (q_1, \dots, q_m)$ by taking
the coordinate  $q_j$ to
coincide with the corresponding
 coordinate $v_j$ of $v$ if it is
 an integer,
 and otherwise we let $q_j$ be an arbitrary
 nonzero integer.
This implies that $q$ is a
nonzero vector in $\Z^m$ which is not
orthogonal to $v$ (with respect to the
unit cube in $\R^m$). 
But now $(u,v)$ 
is orthogonal to all vectors of the
form $(p,q)$, $p \in \Z^n$, since these
vectors are in $B$. However this
 is possible only if $u$
is orthogonal to all  vectors
$p \in \Z^n$, which contradicts
the completeness of the spectrum
$\Z^n$ of the unit cube in $\R^n$.
This shows that  $\Gam$ is indeed
maximal.

It remains to show that  $\Gam$ is  not
complete. Indeed, if it was complete  
then  $\Gam$ would be a tiling set
for the unit cube
 in $\R^n \times \R^m$, 
 again due to the result in
  \cite{lagarias2000orthonormal}, \cite{iosevich1998spectral}, \cite{kolountzakis2000packing}.
In turn this would imply that $\Lam$
is a tiling set for the unit cube
 in $\R^n$. As a consequence,
 $\Lam$ would be a spectrum,
 contrary to our assumption.
\end{proof}

\subsection{}
 Another possibility for using
 lower-dimensional examples 
 in order to construct  examples
 in higher dimensions, is to 
  utilize cartesian products.
 
\begin{thm}    
\label{thm:maxprod}
Let
 $A \sbt \R^n$
 and
 $B \sbt \R^m$
 be two maximal orthogonal sets
for the unit cubes in $\R^n$ and
$\R^m$  respectively. Then the cartesian
product $A \times B$
 is a  maximal orthogonal set
for the unit cube in $\R^n \times \R^m$.
Moreover, $A \times B$ is a complete set if and only if 
both $A$ and $B$ are  complete sets.
\end{thm}

\begin{proof}
It is straightforward to check that
 $A \times B$
 is an  orthogonal set
for the unit cube in $\R^n \times \R^m$.
Suppose  to the contrary that $A \times B$ is not 
 maximal, and let $(u,v) \in \R^n \times \R^m$ 
 be orthogonal to all the points of $A \times B$.
The  maximality of $A$
implies the existence of a point $a \in A$
which is not orthogonal to $u$.
Similarly, by the maximality of $B$
there is $b \in B$ which is not orthogonal to $v$.
It follows that  $(u,v)$
 is not orthogonal to 
 the point 
$(a,b)$ which belongs to $A \times B$,
a contradiction. Hence $A \times B$ must be
a maximal set.

Finally, it is well known,
see e.g.\ \cite[Theorem 4]{JP99},
that if both $A$ and $B$ are  complete sets,
then also $A \times B$ is complete.
The converse is also true, 
see \cite[Lemma 2]{JP99}.
\end{proof}

\subsection{}
It is natural to pose the following question, which generalizes \probref{prob:lines}.
Let us define  the
\emph{affine dimension} of a
discrete set
$\Lam \sbt \R^d$ to be the smallest
integer $k$ such that 
$\Lam$ can be covered by a
 finite number of  translated 
 $k$-dimensional subspaces.

\begin{problem}
How  small can be the
affine dimension of  a  maximal 
orthogonal set
for the unit cube in $\R^d$?
\end{problem}

If we write $d = 3q + r$ where
$r \in \{0,1,2\}$, then the set
$\Gam = \Lam \times \cdots \times \Lam \times \Z^r$
is a maximal incomplete orthogonal set for the unit 
cube in $\R^d$, where  $\Lam \sbt \R^3$
is the set from \thmref{thm:thin},
which appears $q$ times in the product.
We observe that this set $\Gam$
is contained in a finite union of translated
subspaces of dimension $k = 2q + r$.


\section{Some general properties of maximal orthogonal sets}
\label{sec:properties}

\subsection{}
In the previous sections,
we have constructed examples of 
maximal incomplete orthogonal sets
for the unit  cube in $\R^d$, $d \ge 3$.
Moreover, we have seen
  examples of rather ``thin''
maximal orthogonal sets.
A natural question arises as to
  how small can these sets be.
The following proposition 
gives a necessary condition on a maximal set,
which in particular shows that it
cannot be ``too small'' in a sense.

\begin{thm}    
\label{thm:coordinate}
Let $\Lam \sbt \R^d$ be a maximal orthogonal set
for the unit cube. Then for every point
$(\lam_1, \lam_2, \dots, \lam_d)
\in \Lam$ and for every integer $n$,
the number $\lam_1 + n$
must appear as the first
coordinate of some point from $\Lam$.
Similarly, $\lam_2 + n$
must appear as the second
coordinate of some point from $\Lam$,
and so on.
\end{thm}

\begin{proof}
By translating the set $\Lam$ we may assume
that the point 
$(\lam_1, \lam_2, \dots, \lam_d)$ lies at the origin.
By symmetry it also suffices to prove that every integer $n$ must appear as the first coordinate of some point from $\Lam$.  

Suppose to the contrary that $n$ is an integer which does not appear as the first coordinate of any point from $\Lam$. We then claim that the point $\xi = (n,0,0,...,0)$ is  orthogonal to all points of $\Lam$, which contradicts  the maximality of $\Lam$. Indeed, $n$ must be a nonzero integer and hence $\xi$ is orthogonal to the origin. Now let $\mu$ be any point of $\Lam$ other than the origin. Then $\mu$ is orthogonal to the origin and so at least one of the coordinates of $\mu$ is a nonzero integer. If this coordinate is not the first one then $\xi$ is orthogonal to $\mu$ and we are done. Otherwise the first coordinate of $\mu$ is a (nonzero) integer, which cannot be equal to $n$ by assumption. This once again implies that $\xi$ is orthogonal to $\mu$, and so the assertion is established.
\end{proof}

 As an immediate  consequence of \thmref{thm:coordinate}  we obtain:
 
\begin{cor}
\label{cor:infinite}
Any maximal orthogonal set
for the unit cube must be an infinite set.
\end{cor}

Moreover, it follows 
that a maximal orthogonal set 
cannot be ``too localized'' in space.
Specifically,
consider an axis-aligned
slab of width one,
i.e.\ a set of the form
\begin{equation}
\label{eq:axisslab}
 S_j(a) = \{ (x_1, x_2, \dots, x_d) : a \le x_j \le a+1\},
\quad 1 \le j \le d, \quad a \in \R. 
\end{equation}
Since any closed interval of length $1$
contains an integer, 
\thmref{thm:coordinate} implies:
\begin{cor}    
\label{cor:slab}
A maximal orthogonal set
for the unit cube must intersect
 any axis-aligned slab of width one.
\end{cor}

\subsection{}
Assume now that
 $\Lam$ is a \emph{finite} orthogonal set
for the unit cube in $\R^d$. Then 
$\Lam$ is not maximal (\corref{cor:infinite}).
A standard application of Zorn's lemma shows that $\Lam$
 can always be embedded as a subset of a maximal orthogonal set.
Is it always possible to
embed $\Lam$  as a subset of 
an orthogonal and complete set
(i.e.\ a  spectrum) of the   cube?

 \thmref{thm2DIM} tells us that
the answer is   `yes' in two dimensions.
Let us show that it is
 `no' in dimensions $3$ and higher:

\begin{thm}
There exists a finite 
orthogonal set
for the unit cube  in $\R^d$, $d \ge 3$,
which cannot be embedded 
as a subset of any spectrum of the cube.
\end{thm}

\begin{proof}
By our previous results  there exists 
a maximal but incomplete
orthogonal set $\Lam$ 
for the unit cube.
For each $n$ define
$\Lam_n := \Lam \cap [-n,n]^d$.
The set $\Lam_n$ is finite for every $n$.
We claim that there is $n$
such that
the set $\Lam_n$  cannot be embedded 
as a subset of any spectrum.
Indeed, if this is not the case then 
each $\Lam_n$ is contained in some
 spectrum $\Gam_n$ of the cube.
 It is known, 
 see \cite[Section 3]{greenfeld2016spectrality},
 \cite{kolountzakis2016cylinders},
 that there exists a weakly
  convergent subsequence
  $\Gam_{n_j}$ whose weak limit
  $\Gam$ is also a spectrum for the cube.
 We have $\Lam_m \sbt \Lam_{n_j} \sbt \Gam_{n_j}$
 for $n_j > m$, and letting $j \to +\infty$
 it follows that the weak limit $\Gam$
 contains the set $\Lam_m$. Since $m$
 is arbitrary we conclude that
 $\Lam \sbt \Gam$, that is, $\Lam$ can be embedded in a spectrum, in contradiction to its maximality and incompleteness.
\end{proof}


\section{Spectral sets other than the cube}
\label{sec:no-cube}

As we saw in the introduction, in the case of the cube in dimension $1$ (i.e.\ the unit interval), any maximal orthogonal set of frequencies necessarily forms a spectrum. A natural question arises whether this is true for any spectral set $\Om \subset \R$. Below, we utilize an idea of the previous sections to show that this is not the case.

\begin{thm}
\label{thm:intervals}
There exists a spectral set $H \sbt \R$ (a union of finitely many intervals with integer endpoints) which admits a maximal but incomplete orthogonal set.
\end{thm}

We motivate the concrete example given below by a short informal discussion. We are looking for a set $H\subset \R$ such that $H$ is spectral, but there exists a maximal orthogonal set of exponentials in $L^2(H)$ which is not complete. As Fuglede's conjecture is still open in $\R$, we may as well assume that $H$ tiles $\R$ by translations. If so, it is known that any tiling of $\R$ by a bounded region $H$ is periodic
and is essentially equivalent to a tiling of $\Z$
(see \cite{lagarias1996tiling}).
 Furthermore, any tiling of $\Z$ is known to arise as a periodic extension of a tiling of a finite cyclic group $\Z_N = \Z / N\Z$ (see \cite{newman1977tesselations}). 
  It is therefore natural to look for the set $H$ in the form $H=\cup_{j=1}^k [h_j, h_j+1)$, where the set $H_0=\{h_1, \dots, h_k\}\subset \{0, 1, 2, \dots, N-1\}$ is a tile (and a spectral set) in the group $\Z_N$ (where we have identified $\Z_N$ with the set $\{0, 1, 2, \dots, N-1\}$ in the natural way).

However, we note that if the cardinality of the set $H_0$ is a prime number,  then the structure theory of tiling of finite abelian groups indicates that any maximal orthogonal set for $H_0$ is necessarily a spectrum. Presumably the same holds if the number of elements of the set $H_0$ is a prime power, or even a product of two prime powers,
see \cite{coven1999tiling}.
This suggests that we should be looking for a more complicated example.

In the concrete choice of $N$ and the set $H_0\subset \Z_N$ we shall therefore mimic the ``thin'' example of \secref{sec:thin}, given in dimension $3$ for the cube. Let $p<q<r$ be distinct odd primes, and let $N=p^2q^2r^2$. Then $\Z_N$ is isomorphic to the product group $\Z_{p^2} \times \Z_{q^2} \times \Z_{r^2}$, with a group isomorphism being the mapping $\phi:\Z_{p^2} \times \Z_{q^2} \times \Z_{r^2} \to \ZZ_N$ given by
\begin{equation}
\label{eq:iso}
\phi(a, b, c) = q^2r^2 a + p^2r^2 b + p^2q^2 c
\end{equation}
(clearly  this mapping is  a
  homomorphism, and one can easily check that 
  it has a trivial kernel,
    hence it is  an isomorphism).

Let us now consider the ``discrete cube'' $H_0 \subset \Z_{p^2} \times \Z_{q^2} \times \Z_{r^2}$ given by
\begin{equation}
\label{eq:h0}
H_0 = \{(a,b,c): \ 0\le a \le p-1, \ 0\le b \le q-1, \ 0\le c \le r-1\}.
\end{equation}
The Fourier transform $\ft{\1}_{H_0}$ is
 a function on the dual group, which we can 
 also identify with
  $\Z_{p^2} \times \Z_{q^2} \times \Z_{r^2}$.
The zeros of $\ft{\1}_{H_0}$ are then the vectors 
whose first coordinate is a nonzero multiple of $p$, or the second coordinate is a nonzero multiple of $q$, or the third coordinate is a nonzero multiple of $r$. As a consequence, $H_0$ is a spectral set in the group $\Z_{p^2} \times \Z_{q^2} \times \Z_{r^2}$ and the set
$\Gamma_0=\{(u, v, w): \ p|u, \ q|v, \ r|w\}$  serves as a spectrum for $H_0$. Indeed, $\Gam_0$ is an orthogonal set of frequencies for $H_0$ and it has the same cardinality as $H_0$  (see also \cite{agora2018spectra}).

We will now exhibit a maximal but incomplete set of frequencies for $H_0$. We let $\Lambda_0\subset \Z_{p^2} \times \Z_{q^2} \times \Z_{r^2}$ be the set of all vectors that have one of the following forms:
\begin{equation}
(0,0,0), \quad (n, 1-k, 1),  \quad (1, k, 1-m), \quad (1-n, 1, m)
\end{equation}
where $n$ is a nonzero multiple of $p$, $k$ is a nonzero multiple of $q$ and $m$ is a nonzero multiple of $r$. 
(Notice that these elements of $\Lam_0$
are analogous to those in \eqref{eq:thinpoints}.)
It is then straightforward to check that $\Lambda_0$ forms an orthogonal set of frequencies for $H_0$, and one can repeat the proof of 
\secref{sec:thin} verbatim to show that $\Lambda_0$ is maximal. It is also clear that $\Lambda_0$ is not a spectrum of $H_0$, since we have $|H_0|=pqr$ while
\begin{equation}
|\Lambda_0| < pq+qr+rp < 3 q r \le pqr.
\end{equation}

Via the isomorphism \eqref{eq:iso} we deduce that $\phi(H_0) \subset \ZZ_N$ is a spectral set which has a maximal, incomplete set of frequencies $\Lambda_0'$. We view $\Lambda_0'$ as a subset of $\ZZ_N$ via the identification of $\ZZ_N$ with its dual group.

Finally, let $H \sbt \R$ be the set defined by $H=\cup_{h \in \phi(H_0)} [h, h+1)$, or equivalently, 
$\1_H = \1_{\phi(H_0)} \ast \1_{[0, 1)}$. This implies that the zero set of the Fourier transform $\ft{\1}_H$ satisfies
\begin{equation}
\label{new-zeros}
  \big(\{\ft{\1}_H = 0\} \cup \{ 0\}\big) 
\cap \frac{1}{N}\ZZ
 = \frac{1}{N}\big( 
 \{\ft{\1}_{\phi(H_0)}=0 \}  \cup \{ 0\} \big) + \ZZ,
\end{equation}
where we view $\{\ft{\1}_{\phi(H_0)}=0 \}$ as a subset of $\{0, 1, 2, \ldots, N-1\}$. Define also
\begin{equation}
\Lambda = \frac{1}{N}\Lambda_0'+\ZZ
\end{equation}
and observe that this is an orthogonal set for $H$ because of \eqref{new-zeros}. For the same reason we have that if $\Gam'_0$ is a spectrum of $\phi(H_0)$ in $\ZZ_N$ then $\Gam = \frac1{N} \Gam'_0 +\ZZ$ is a spectrum of $H$ in $\RR$. The orthogonality of $\Gam$ follows from \eqref{new-zeros} and the completeness follows from (a) having the right density and (b) being periodic, as (a) and (b) together imply that the ``packing'' $|\ft{\1}_H|^2+ \Gam$ is actually a tiling
(see \cite[Section 2]{kolountzakis2024orthogonal}).

It remains to show the maximality of $\Lambda$ (the incompleteness of $\Lam$ is guaranteed since its density is too small). Assume   that $\Lambda \cup \{x\}$ is an orthogonal set of frequencies for $H$, where $x \in \RR\setminus\Lambda$. By the $\ZZ$-periodicity of $\Lambda$ we may assume that $x \in [0, 1)$.
Now, if we knew that 
$x=k/N$ for some $k\in\{0,1, 2, \ldots, N-1\} \setminus\Lambda'_0$, then again by
\eqref{new-zeros} this would
imply that $\Lambda'_0 \cup \{k\}$ is an orthogonal set of frequencies for $\phi(H_0)$ in $\ZZ_N$, which contradicts the maximality of $\Lambda_0'$.
To conclude the proof we  therefore need to rule out the possibility that $x$ is a real number in $[0,1)$ which is not of the form $k/N$.

To this end we will need the following lemma.

\begin{lemma}
\label{few-zeros}
If $H_0 \subset \ZZ_{p^2}\times\ZZ_{q^2}\times\ZZ_{r^2}$ is the set defined by \eqref{eq:h0} and $\phi$ is given by \eqref{eq:iso}
then, viewing the set $\phi(H_0)
\sbt \{0, 1, 2, \dots, N-1\}$
 as a subset of the integer group $\Z$, the Fourier transform of its indicator function (viewed as a function on the dual group $\TT =  \R / \Z$) has no zeros outside the set
\begin{equation}
\frac{1}{p^2q^2r^2}\ZZ.    
\end{equation}
\end{lemma}

\begin{proof}
We must  show  that if 
 $\xi \in \RR$ and
 $\ft{\1}_{\phi(H_0)}(\xi/N)=0$ then $\xi \in\ZZ$.

 Let   $A = \{0, 1, \ldots, p-1\}$,
 $B = \{0, 1, \ldots, q-1\}$
 and  $C = \{0, 1, \ldots, r-1\}$
 viewed as subsets of the integers. 
 We make the following observation:
  if we interpret 
 the expression  \eqref{eq:iso}
 as an integer (and not as an
 element of $\Z_N$),
 then all the values attained by
 this expression for 
 $a \in A$, $b \in B$ and $c \in C$
 lie in the subset
 $\{0,1,2,\dots, N-1\}$
  of the integers. Indeed,
 all these values are nonnegative and
 not greater than
\ifdefined\SMART
\begin{align*}
   q^2r^2 p + p^2r^2 q + p^2q^2 r
      &= pqr(qr + pr + pq) \\
      &< pqr(3qr)\\
       &\le p^2 q^2 r^2\\
       &= N.
\end{align*}
\else
\begin{equation}
  q^2r^2 p + p^2r^2 q + p^2q^2 r
     = pqr(qr + pr + pq) 
     < pqr(3qr) \le p^2 q^2 r^2 = N.
\end{equation}
\fi
This means that as the triple $(a,b,c)$
 goes through the elements of
 $A \times B \times C$, then
  the expression  \eqref{eq:iso}
not only goes through the 
 elements of   $\phi(H_0)$
 viewed as congruence classes
 modulo $N$,  but it also
goes through the representatives 
 in  $\{0,1,2,\dots, N-1\}$
 (as a subset of $\Z$)
 of the  elements of the set
 $\phi(H_0)$.

This observation allows us to compute the 
Fourier transform
 $\ft{\1}_{\phi(H_0)}(\xi/N)$.
 Assume that $\xi \in \R$
 is not an integer, then we have
\begin{align}
\ft{\1}_{\phi(H_0)}(\xi/N) &= \sum_{a \in A} \sum_{b \in B} \sum_{c \in C} \exp(-2\pi i \phi(a, b, c)\xi/N) 
\label{eq:ftcalc1}\\
 &= \sum_{a \in A} \sum_{b \in B} \sum_{c \in C} \exp(-2\pi i (q^2r^2a+p^2r^2b+p^2q^2c)\xi/N)
 \label{eq:ftcalc2}\\
 &= \sum_{a=0}^{p-1} e^{-2\pi i a \xi/p^2}
 \sum_{b=0}^{q-1} e^{-2\pi i b \xi/q^2}
 \sum_{c=0}^{r-1} e^{-2\pi i c \xi/r^2}
\label{eq:ftcalc3}\\
 &= \frac{1-e^{-2\pi i \xi/p}}{1-e^{-2\pi i \xi/p^2}} \cdot \frac{1-e^{-2\pi i \xi/q}}{1-e^{-2\pi i \xi/q^2}} \cdot \frac{1-e^{-2\pi i \xi/r}}{1-e^{-2\pi i \xi/r^2}},
 \label{eq:ftcalc4}
\end{align}
 where the observation made above was used
 in 
 \eqref{eq:ftcalc1} and \eqref{eq:ftcalc2}.
It remains to notice that the   quantity
in \eqref{eq:ftcalc4}
is nonzero, as otherwise $\xi$ would have to be an integer divisible by $p$, $q$ or $r$. Hence the lemma
is established.
\end{proof}

Finally, we use the lemma to
finish the proof of the 
 maximality of $\Lambda$.  Recall that we have
 assumed that $\Lambda \cup \{x\}$ is an orthogonal 
 set  of frequencies for $H$, where $x$ is a
 real number in $[0,1) \setminus\Lambda$. 
Since $0 \in \Lambda$, Lemma \ref{few-zeros} 
allows us to conclude that $x=k/N$ for some 
$k\in\{1, 2, \ldots, N-1\}\setminus\Lambda'_0$. 
As we have seen,
 this implies that $\Lambda'_0 \cup \{k\}$ is 
 an orthogonal set of frequencies
 for $\phi(H_0) \sbt \ZZ_N$,
  contradicting the maximality of $\Lambda_0'$. 

\subsection*{Note added in proof} After submitting our paper we learned that maximal orthogonal sets for the unit cube in low dimensions were studied in the paper \cite{lysakowska2011structure}, where they are referred to as sets which determine an unextendible system of cubes. In particular \cite[Theorem 5]{lysakowska2011structure} gives a complete structural description of the maximal incomplete orthogonal sets for the unit cube in three dimensions. One can infer from this result a negative answer to Problem 4.2.


\end{document}